\documentclass{amsart}

\usepackage{color}

\numberwithin{equation}{section}
\def\di{\displaystyle}
\def\O{\Omega}

\def\S{\Sigma} 
\def\n{\nabla}
\def\nb{\bar\nabla}
\def\P{\partial}

\def\bn{\bar \nabla}
\def\bd{\bar \Delta}
\def\a{\alpha}
\def\b{\beta}
\def\n{\nabla}

\def\tr{\mathrm{tr}}
\def\o{\omega}
\def\O{\Omega}

\def\a{\alpha}
\def\b{\beta}
\def\g{\gamma}
\def\d{\delta}

\def\l{\lambda}

\def\De{\Delta}
\def\n{\nabla}
\def\<{\langle}
\def\>{\rangle}
\def\div{{\rm div}}
\def\De{\Delta}
\def\n{\nabla}

\def\nn{\nonumber\\}

\def\tr{\mathrm{tr}}
\def\o{\omega}
\def\O{\Omega}

\def\a{\alpha}
\def\b{\beta}
\def\g{\gamma}
\def\d{\delta}
\def\l{\lambda}

\newtheorem{theorem}{Theorem}[section]
\newtheorem{lemm}[theorem]{Lemma}

\newtheorem{defi}[theorem]{Definition}
\newtheorem{theo}[theorem]{Theorem}
\newtheorem{prop}[theorem]{Proposition}

\newtheorem{rema}[theorem]{Remark}
\newtheorem{coro}[theorem]{Corollary}

\numberwithin{equation}{section}

\begin{document}
\title[An integral formula for affine connections]{An integral formula for affine connections}
\author{Junfang Li and Chao Xia}

\address{Department of Mathematics\\
University of Alabama at Birmingham\\
Birmingham, AL 35294}
\email{jfli@uab.edu}

\address{School of Mathematical Sciences\\
Xiamen University\\
361005, Xiamen, P.R. China}
\email{chaoxia@xmu.edu.cn}
\thanks{Research of CX is  supported in part by NSFC (Grant No. 11501480) and the Fundamental Research Funds for the Central Universities (Grant No. 20720150012). }
%    author two information

%    \subjclass is required.
%\subjclass[2010]{Primary }

\date{}

\dedicatory{}

%    "Communicated by" -- provide editor's name; required.
\commby{}

%    Abstract is required.
\begin{abstract}
    In this article, we introduce a $2$-parameter family of affine connections and derive the Ricci curvature. We first establish an integral Bochner technique. On one hand, this technique yields a new proof to our recent work in \cite{LX} for substatic manifolds. On the other hand, this technique leads to various geometric inequalities and eigenvalue estimates under a much more general Ricci curvature conditions. The new Ricci curvature condition interpolates between static Ricci tensor and $1$-Bakry-\'Emery Ricci, and also includes the conformal Ricci as an intermediate case. 
\end{abstract}

\maketitle

\section{Introduction}

The classical Reilly formula is actually an integral Bochner formula for gradient vector fields on manifolds with boundary in references. It has been proven to be a quite useful tool in differential geometry. 

Motivated by a work of Qiu and the second author \cite{QX}, we have established a generalized Reilly type formula in previous work \cite{LX}. Such a generalization enabled us to prove a Heintze-Karcher-Ros-Brendle type inequality under a sub-static condition. Such kind of inequality, which could lead to an Alexandrov type rigidity theorem, has been proved first by Brendle \cite{Br}. See also recent work by Wang-Wang \cite{WW}. Moreover, the general formula has been used to prove several geometric inequalities in \cite{LX} and also applied by Chen-Wang-Wang-Yau \cite{CWWY} to prove the stability of quasi-local energy with respect to a static spacetime.

The formula was proved via very basic integration-by-parts with respect to the Levi-Civita connection, although the computation is complicated and tedious. The key point is that we introduced a ``weight'' function $V$, which was motivated by Brendle and Brende-Hung-Wang \cite{Br, BHW}.

In this article, we  adapt a new point of view  to recover the formula in \cite{LX}. We find that  the formula in \cite{LX} is indeed an integral Bochner formula for some special vector fields with respect to a special {\it torsion-free affine }connection instead of the Levi-Civita connection. 
Moreover, this turns out to be a general phenomenon that a wide class of {\it torsion-free affine }connections give rise to a class of Reilly type formulas.

\

Let $(M^n, \bar g)$ be an $n$-dimensional smooth Riemannian manifold and $\nb$ be the Levi-Civita connection of $\bar g$. Let $V=e^u$ be a positive smooth function on $M$, where $u$ is a smooth function on $M$. We call $(M, \bar g, V)$ a Riemannian triple.  

For $\a, \g\in \mathbb{R}$, we define a $2$-parameter family of affine connections: given two vector fields $X, Y$ on $M$, let
\begin{eqnarray}
D^{\a,\g}_XY=%&\nb_XY+\a V^{-1}dV(X)Y+\a V^{-1}dV(Y)X +\g \bar g(X,Y)V^{-1}\nb V\nonumber
\nb_XY+\a du(X)Y+\a du(Y)X +\g \bar g(X,Y)\nb u.
        \label{defi D}
\end{eqnarray}
For simplicity, we often omit the superscript $\a,\g$ when it is clear in the context.
The Ricci curvature under $D^{\a,\g}$ is (see Proposition \ref{thm Ricci} below)
\begin{align}  \label{Ricci}
Ric^{D}:=&Ric-[(n-1)\a+\g]\bar\nabla^2 u+\big[(n-1)\a^2-\g^2\big]du\otimes du\nonumber\\
        &+\left[\g \bar\De u +(\g^2+(n-1)\a\g)|\nb u|^2\right]\bar g.
  \end{align}
  
There are two trivial cases among all $D^{\a,\g}$.  One is the Levi-Civita connection for $\bar g$ when $\a=\g=0$, while the other is  the Levi-Civita connection for the conformal metric $e^{2\a u} \bar g$ when $\a=-\g$. For other choices of $\a$ and  $\g$, $D^{\a,\g}$ may not be a Levi-Civita connection for any Riemannian metric.

For the case $\a=0, \g=1$, one sees from \eqref{Ricci} that
\begin{eqnarray}
Ric^{D^{0,1}}=  Ric-\frac{\bar \nabla^2 V}{V}+\frac{\bar \Delta V}{V}\bar g,
\end{eqnarray}
where $Ric$ is the usual Ricci curvature for $\bar g$. We call $Ric^{D^{0,1}}$ {\em static Ricci tensor}. A Riemannian triple $(M, \bar g, V)$ satisfying $VRic^{D^{0,1}}=0$ is referred to a static manifold in the literature, see e.g. \cite{Cor}.

For the case $\a=\frac{1}{n-1}, \g=0$, one  sees from \eqref{Ricci} that
\begin{eqnarray}
Ric^{D^{\frac{1}{n-1},0}}=  Ric-\mathrm{\bar \nabla^2} u +\frac{1}{n-1}d u\otimes d u.
\end{eqnarray}
This is in fact the {\em $1$-Bakry-\'Emery Ricci tensor} in the literature which was introduced by Bakry and \'Emery \cite{BaE}. The fact that this affine connection gives rise to the $1$-Bakry-\'Emery Ricci tensor  has also been observed by Wylie-Yeroshkin \cite{WY} in their recent studies of manifolds with density.

\

 The main result of this paper is the following Reilly type integral formula.

\begin{theo} \label{theo RF0}    
Let $(M^n, \bar g, V=e^u)$ be an $n$-dimensional Riemannian triple and $\a,\g\in \mathbb{R}$. Let   $D=D^{\a,\g}$ be the  affine connection defined as in (\ref{defi D}) and $\tau=(n+1)\a+\g$. Let $\phi$ be a smooth function on a bounded domain $\O\subset M$ with smooth boundary $\S$. Then the following integral formula holds:
         \begin{align}         \label{formula RF 1}
        &\di\int_{\O}V^\tau \left[\left|\bd^D \phi\right|^2-\left|\bn^{2,D} \phi\right|_{\bar g}^2 \right] -V^\tau Ric^{D}( \nb^D \phi, \nb^D \phi) d\O \\ =&\di\int_{\S}V^{\tau}\left[H^D\<\bn^D \phi,\nu\>^2+\left(h-\g u_\nu g\right)(\n^D\phi, \n^D \phi)-2V^{-\g}\<\n^D \phi, \n^D(V^\g \phi_\nu)\>\right]dA.\nonumber
    \end{align}
\end{theo}

The notations $\nb^D$, $\bn^{2,D}$ and $\bd^D$ play the role of gradient, Hessian and Laplacian with respect to $D$, the exact definition will be given in Definition \ref{notation}. $H^D:=H+(n-1)\a u_\nu$ is the affine mean curvature, where $H$ is the usual mean curvature.\\

Formula \eqref{formula RF 1} reduces to Reilly's original formula for $\bar g$ in the case $\a=\g=0$ and for $e^{2\a u}\bar g$ in the case $\a=-\g$. 
Moreover, when $\a=0, \g=1$, it reduces to the following
\begin{theo}
   Let $(M^n, \bar g, V=e^u)$ be an $n$-dimensional Riemannian triple.  Let $\phi$ be a smooth function on a bounded domain $\O\subset M$ with smooth boundary $\S$. We have
    \begin{align}  \label{theo RF sub}
        &\di\int_{\O}V^3\Big[(\bar\De \phi+2\bar\n u\bar\n \phi)^2-|\bar\n_i\bar\n_j\phi+\bar\n_i u\bar\n\phi_j+\bar\n_j u\bar\n_i\phi|^2\Big]d\O\\
        =&\di\int_{\S}V^{3}\Big(H\phi_\nu^2+\left(h-u_\nu g\right)(\n\phi, \n\phi)-2 V^{-1}\<\n \phi, \n (V\phi_\nu)\>\Big)dA\nonumber\\
        &+\int_{\O}V^{3}(Ric-\frac{\nb^2 V}{V}+ \frac{\bar\De V}{V}\bar g)(\nb\phi,\nb\phi)d\O.\nonumber
          \end{align}
\end{theo}
If we let $\phi=\frac{f}{V}$ in \eqref{theo RF sub}, then we recover Theorem 1.1  in \cite{LX} by a different method. 

\

%\begin{rema} The formula \eqref{formula RF 1} reduces to \\
% {\bf a. }Reilly's original formula for $\bar g$ in the case $\a=\g=0$ and for $e^{2\a u}\bar g$ in the case $\a=-\g$;\\
%{\bf b.} our previous formula in \cite{LX}  in the case $\a=0, \g=1$, see Appendix \ref{appenB};\\
%{\bf c.} the rest cases have not been discussed in the literature.
%\end{rema}

Let us illustrate the basic idea of the proof of Theorem \ref{theo RF0}. It is well known that a local Bochner formula holds for general vector fields under an affine connection. Since the connection is in general not metric compatible, we have only a divergent form instead of the Laplacian of some function in the Bochner identity, see Proposition \ref{Bochner formula}. Nevertheless, we are able integrate this local Bochner formula to get an integral formula. To achieve an effective Reilly type formula, there are two innovative ingredients with this method. First, we choose a right volume form, which is a ``weight'', to get the divergent-free property. Second, we choose a suitable vector field $X$ which satisfies $DX$ is symmetric. It turns out that we choose $X=\bn^D \phi$ and the volume form $V^\tau d\O$, see Lemmata \ref{sym vf} and \ref{weight}.

\

With the integral formula in Theorem \ref{theo RF0} at hand, we are able to prove Heintze-Karcher type,  Minkowski type, and  Lichnerowicz type inequalities.
\begin{theo}\label{thm ineq}
Let $(M^n, \bar g, V=e^u)$ be an $n$-dimensional Riemannian triple and $\a,\g\in \mathbb{R}$. Let   $D=D^{\a,\g}$ be the  affine connection defined as in (\ref{defi D}) and $\tau=(n+1)\a+\g$.  Then we have the following results.
%Assume $V>0$ in $\bar\O$ for simplicity and let $V=e^u$. 
\begin{itemize}\item[(i)] {\bf Heintze-Karcher type inequality:} if $Ric^{D}\ge 0$ and $H^D>0$, then
\begin{equation}
         n\di\int_{\O}V^\tau d\O\le\di (n-1)\int_{\S}\frac{V^\tau}{H^D} dA.
         \label{Heintze-Karcher}
     \end{equation}
Equality in \eqref{Heintze-Karcher} holds only if $\S$ is umbilical.

\item[(ii)] {\bf Minkowski inequality:} If $Ric^{D}\ge 0$ and $h-\g u_\nu g> 0$, then
 \begin{eqnarray}\label{rm2}
        \left(\int_\S V^{\tau-\a} dA\right)^2\geq \frac{n}{n-1}\int_\Omega V^\tau d\O \int_\S H^DV^{\tau-2\a} dA,
    \end{eqnarray}
Equality in \eqref{rm2} holds only if $\S$ is umbilical.

\item[(iii)] {\bf Lichnerowicz inequality:} If $Ric^{D}\ge (n-1)V^{\a-\g} \bar g$ and \\
{\bf a) }$\S=\emptyset$, then $\lambda_1(\bd^D)\ge n;$\\
{\bf b) }$\S \neq\emptyset$ and $\S$ satisfies $H^D\ge 0$, then $\lambda_1^{Dir}(\bd^D)\ge n;$\\
{\bf c) }$\S\neq\emptyset$ and $h-\g u_\nu g\ge 0$, then $\lambda_1^{Neu}(\bd^D)\ge n.$
\end{itemize}
\end{theo}

Here $\lambda_1, \lambda_1^{Dir}$ and $\lambda_1^{Neu}$ indicate the closed, the Dirichlet and the Neumann first (nonzero) eigenvalue of the affine Laplacian $\bd^D$, i.e., there exists some non-trivial $\phi$ such that $\bd^D \phi=-\l_1\phi$ with Dirichlet boundary condition $\phi=0$ or Neumann boundary condition $\phi_\nu=0$.

\begin{rema} \
\begin{itemize}
    \item[(i)] Theorem \ref{thm ineq} reduces to  Heintze-Karcher, Minkowski, and Lichnerowicz  inequalities for $\bar g$ in the case $\a=\g=0$ or $e^{2\a u} \bar g$ in the case $\a=-\g$. See Section 2.1 for an overview.

\item[(ii)]  In the case $\a=0, \g=1$, the Heintze-Karcher type inequalities were first proved by Brendle \cite{Br}, then by Wang-Wang \cite{WW} for non-homologous static manifolds without warped product structure by using the same method and by the authors for general cases using Reilly type formulas in \cite{LX}. The Minkowski type inequalities have been proved in our previous work \cite{LX}.

\item[(iii)] Theorem \ref{thm ineq} gives new geometric inequalities under the condition of nonnegative $1$-Bakry-\'Emery Ricci, which is the case $\a=\frac{1}{n-1}, \g=0$. To illustrate the idea, we only list the example of the Heintze-Karcher type inequality and the others hold true similarly.

\noindent{\bf Corollary.} {\em   Let $(M^n, \bar g, e^ud\O)$ be a smooth weighted Riemannian manifold and $\O$ be a bounded domain in $M$. If the $1$-Bakry-\'Emery Ricci is nonnegative, namely,
    \begin{align}
    &Ric-\bar\nabla^2 u+\frac{1}{n-1}du\otimes du\nonumber\ge 0,
        \label{Ric nonnegative tilde 1} 
    \end{align}
    and the weighted mean curvature $H+\<\bar\nabla u,\nu\>>0$,
    then the following inequality holds: 
     \begin{equation}
         n\di\int_{\O}e^{\tau u}d\O\le\di (n-1)\int_{\S}\frac{e^{\tau u}}{H+\<\bar\nabla u,\nu\>} dA,
         \label{Heintze-Karcher coro 1}
     \end{equation}
     where $\tau=\frac{n+1}{n-1}$. Moreover, if equality in \eqref{Heintze-Karcher} holds, then $\S$ is umbilical.  
}

We remark that the weight in \eqref{Heintze-Karcher coro 1} is $e^{\tau u}$ instead of $e^u$. The new weight volume form $e^{\tau u}d\O$ has a property that it is parallel under the affine connection $D^{\a,\g}$. In the special case of $1$-Bakry-\'Emery Ricci curvature, this parallel property has been observed by Wylie-Yeroshkin \cite{WY} recently.
\end{itemize}
\end{rema}

In particular, for  a Riemannian triple  $(M, g, V)$ whose static Ricci tensor has a positive lower bound, we get the first eigenvalue estimate for the operator $V\bar \De - \bd V\cdot$.

\begin{coro}\label{coro}
Let $(M^n, \bar g, V)$ be an $n$-dimensional Riemannian triple.  Let $\O$ be a bounded domain in $M$ with smooth boundary $\S$. Assume the static Ricci tensor  satisfies
$$V Ric- \mathrm{\bar \nabla^2}  V+ \bar \De V \bar g \geq (n-1) \bar g.$$ 
Then we have \\
{\bf a) } if $\S=\emptyset$, then $\lambda_1(V\bar \De - \bd V\cdot)\ge n;$\\
{\bf b) } if $\S \neq\emptyset$ and $\S$ satisfies $H^D\ge 0$, then $\lambda_1^{Dir}(V\bar \De - \bd V\cdot)\ge n;$\\
{\bf c) } if $\S\neq\emptyset$ and $h-\g u_\nu g\ge 0$, then $\lambda_1^{Neu}(V\bar \De - \bd V\cdot)\ge n.$
\end{coro}

\
 
The rest of the paper is organized as follows. In section \ref{sec pre}, we recall classical results, introduce our affine connections, fix the notations, and give the Ricci curvature under affine connections. In section \ref{sec Bochner}, we establish the Bochner formula and prove the main theorem, Theorem \ref{theo RF0}. In section \ref{sec inequalities}, we prove the Heintze-Karcher type and the Minkowski type inequalities of Theorem \ref{thm ineq}. In section \ref{sec Lich}, we prove the Poincare type and the Lichnerowicz type inequalities. In the {\bf Appendix}, we prove Proposition \ref{thm Ricci}. \  

\

\section{Preliminaries and notations}\label{sec pre}

\subsection{Classical results.}
Let us first recall the classical Reilly's formula \cite{Re1}. 
For a bounded domain  $\O$ with boundary $\S$ in an $n$-dimensional smooth Riemannian manifold $(M, \bar g)$  and $\phi\in C^\infty(\bar \O)$, the classical Reilly's formula reads as
\begin{eqnarray}
&&\int_\O (\bd \phi)^2-|\bn^2 \phi|^2 -Ric(\bn \phi, \bn \phi) d\O=\int_\S H\phi_\nu^2+h(\n \phi, \n \phi)+ 2\phi_\nu\De \phi dA.
\end{eqnarray}
Here and throughout this paper, $\bn$ and $\bd$ denote the gradient and the Laplacian on $\O$ respectively, $\n$ and $\De$ denote the gradient and the Laplacian on $\S$ respectively with respect to the induced metric from $\bar g$. $d\O$ and $dA$ are the Riemannian volume form of $\bar g$ and the induced area element from $\bar g$ respectively. $\nu$ is the normal vector field of $\S$ and $\phi_\nu=\nb_\nu\phi$ is the normal derivative of $\phi$. $h(X,Y)=\<\nb_X \nu, Y\>$ is the classical second fundamental form of $\S$ and $H= \tr_{\bar g} h$ is the usual mean curvature of $\S$.

With the Reilly formula, some classical geometric inequalities can be readily proved.
 
\noindent{\bf Heintze-Karcher inequality:}  If $Ric\ge 0$ and $\S$ is mean convex, i.e., $H>0$, then
\begin{eqnarray}\label{HK0}
\int_\S \frac{1}{H}dA\ge  \frac{n}{n-1} Vol(\O).
\end{eqnarray}
 
\noindent{\bf Minkowski inequality:}  If $Ric\ge 0$ and $\S$ is convex, i.e., $h \ge 0$, then
\begin{eqnarray}\label{M0}
Area(\S)^2 \ge  \frac{n}{n-1} Vol(\O)\int_\S H dA.
\end{eqnarray}

\noindent{\bf Lichnerowicz inequality:}  If $Ric\ge (n-1)\bar g$ and $\S$ is empty, then
\begin{eqnarray}\label{L0}
\l_1(\bd)\ge n.
\end{eqnarray}

Similar inequalities like \eqref{HK0} were first derived by  Heintze-Karcher \cite{HK} using the classical approach of Jacobian fields from Riemanian geometry. Ros \cite{Ros} proved the current form of this inequality using Reilly's formula. Inequality \eqref{M0} was first derived by Minkowski \cite{Min} in the Euclidean case as a consequence of the famous Brunn-Minkowski theorem in convex geometry. Reilly \cite{Re3} proved this inequality under the condition of nonnegative Ricci by using his formula. Recently, Wang-Zhang \cite{WZ} gave an alternative proof of Minkowski inequality (\ref{M0}) using ABP method. Inequality \eqref{L0} was proved by Lichnerowicz \cite{Lic} using the classical Bochner technique.

\subsection{Notations under affine connections}\label{subsec Ricci}

As in the introduction, a two parameter family of {\it torsion free} affine connections $D^{\a, \g}$ is defined on $M$ for $\a,  \g\in \mathbb{R}$:
\begin{eqnarray*}
D^{\a,\g}_XY=%&\nb_XY+\a V^{-1}dV(X)Y+\a V^{-1}dV(Y)X +\g \bar g(X,Y)V^{-1}\nb V\nonumber
\nb_XY+\a du(X)Y+\a du(Y)X +\g \bar g(X,Y)\nb u.
\end{eqnarray*}
One checks directly that $D^{\a,\g}$ is torsion-free. For a general affine connection, we  adapt the following convention of Ricci curvature.
\begin{defi}\label{defi Ricci}
    Given an affine connection $D$, for any vector fields $X,Y$, we define the Ricci curvature as
   \begin{align*}
       Ric^D(X,Y)=\o^i\left(R^D(e_i,X)Y\right),
   \end{align*}
   where   $\{e_i\}$ is a local frame of the tangent bundle, $\{\o^i\}$ is the dual $1$-form of $\{e_i\}$ and the Riemann curvature operator $R^D$ is defined as
   \begin{align*}
       R^D(X,Y)Z= D_{X}D_YZ-D_YD_{X}Z-D_{[X,Y]}Z.
   \end{align*}
\end{defi}
\begin{rema}
In the case we have a Riemannian metric $\bar g$, the Ricci curvature $Ric^D$ of $D$ can also be interpreted as
   \begin{align*}
       Ric^D(X,Y)=\<R^D(e_i,X)Y, e_i\>,   \end{align*}
 where   $\{e_i\}$ is an orthonormal frame of the tangent bundle.  
\end{rema}
By direct computation, we derive the following representation of $Ric^D$ for $D=D^{\a,\g}$ in terms of Levi-Civita connection $\bar\n$.
\begin{prop}  \label{thm Ricci} 
    The Ricci curvature $Ric^D$ of $D=D^{\a,\g}$ satisfies the following identity:
    \begin{align*}
       Ric^D=&Ric-[(n-1)\a+\g]\bar\nabla^2 u+[(n-1)\a^2-\g^2]du\otimes du\nonumber\\
        &+\left\{\g \bar\De u +\g[(n-1)\a+\g]|\nb u|^2\right\}\bar g.
    \end{align*}
\end{prop}
We will prove Proposition \ref {thm Ricci} in {\bf Appendix}. As already mentioned in the introduction,  the Ricci curvature $Ric^{D}$ of the new affine connection $D^{\a,\g}$ not only yields new Ricci curvature tensors, but also recovers some of known examples in the literature, which includes  conformal Ricci tensor from conformal geometry, and static Ricci tensor raised from General Relativity and the $1$-Bakry-\'Eemery Ricci tensor from  manifolds with density.  

\
Next we explain the notations in \eqref{formula RF 1}. 
\begin{defi}\label{notation} \
\begin{itemize}
\item[(i)]  The $D$-gradient on $\O$ and $\S$ are defined respectively by $$\bn^D \phi:= V^{\g-\a}\bn \phi, \quad \n^D \phi:= V^{\g-\a}\n \phi$$
\item[(ii)] The $D$-Hessian $\bn^{2,D} \phi$ and $D$-Laplacian $\bd^{D} \phi$ on $\O$ are defined respectively by
\begin{eqnarray}
\bn^{2,D} \phi&:=& D(V^{\g-\a} \nb \phi)\\&=& V^{\g-\a} \left[\bn^2 \phi+\g du\otimes d\phi+\g d\phi \otimes du+ \a \<\bn u, \bn \phi\>\bar g\right], \nonumber
\end{eqnarray}
and
\begin{eqnarray}
\bd^{D} \phi:= \tr_{\bar g} (\bn^{2,D} \phi)= V^{\g-\a} \left[\bd \phi+(2\g+ n\a)\<\bn u, \bn \phi\>\right]. \nonumber
\end{eqnarray}

\end{itemize}
\end{defi}

 We note that in the case $\a=1, \g=-1$, the  $D$-gradient, the $D$-Hessian and the $D$-Laplacian are in consistence with the classical ones for conformal metric $e^{2u}\bar g$. By virtue of this, we believe Definition \ref{notation} is natural for $D^{\a,\g}$.

\

\section{Bochner technique for general affine connections}\label{sec Bochner}
In this section, we review a Bochner identity for general affine connection and prove Theorem \ref{theo RF0}. 

It is well known that under Levi-Civita connection $\nb$, the following Bochner formula holds: for a smooth vector field $X$ on $M$ with the property that $\nb X$ is symmetric, \begin{eqnarray*}
\bar \De \frac12|X|^2=|\bar \n X|^2+ \bar \nabla_X (\div_{\bar g} X)+Ric(X, X),
\end{eqnarray*}
see e.g. Petersen \cite{Petersen} Proposition 33, page 207.

Under an affine connection, the following Ricci identity holds. 
\begin{lemm}{\bf (Ricci identity)}\label{Ricci iden}
 Under local coordinates $\{\P_i\}$, for any smooth vector field $X$, we have
    \begin{align}
        D_iD_jX^i=D_jD_iX^i+R^{D}_{jk}X^k.
        \label{Ricci identity}
    \end{align}
\end{lemm}

\begin{proof}
A vector filed $X$ can be viewed as a $(1,0)$-tensor field and we have
    \begin{align}
        D^2X(V,W)=&D_W(D_VX)-D_{D_WV}X,\nn
        D^2X(W,V)=&D_V(D_WX)-D_{D_VW}X.
        \label{Ricci identity proof equ 1}
    \end{align}
It follows  that
    \begin{align}
        D^2X(W,V)-D^2X(V,W)=&R^{D}(V,W)X.
        \label{Ricci identity proof equ 2}
    \end{align}
Let $\{e_i\}$ be an local frame of the tangent bundle, $\{\o^i\}$ is the dual $1$-form of $\{e_i\}$.    
Then we have from \eqref{Ricci identity proof equ 2}
        \begin{align}
       \o^i\left(D^2X(e_j,e_i)\right)- \o^i\left(D^2X(e_i,e_j)\right)=&Ric^{D}(e_j, X).
    \end{align}
    We finish the proof.
\end{proof}

\begin{prop}\label{Bochner formula}
    {\bf (Bochner formula)}  Let $D$ be an affine connection on $M$ and $Ric^D$ be the Ricci curvature of $D$. Let $X$ be a smooth vector field on $M$. %whose first covariant derivative is symmetric, i.e., $D_iX^j=D_jX^i$. 
Then we have
\begin{align}
    \div^D (D_XX)=(DX)^t \cdot DX+d(\div^D X)(X)+Ric^{D}(X, X).
    \label{Bochner}
\end{align} 
where $div^D$ is  the divergence operator w.r.t. $D$, $\div^D Y=D_i Y^i$ for a vector field $Y$, $(DX)^t$ is the transpose of $D X$.

Under local coordinates $\{\P_i\}$, \eqref{Bochner} reads as
\begin{align}
    D_i(X^jD_jX^i)=&D_iX^j D_jX^i+X^jD_jD_iX^i+R^{D}_{ji}X^iX^j.
    \label{Bochner1}
\end{align} 
Moreover,
\begin{align}
    D_i(X^jD_jX^i-X^iD_jX^j) =&D_iX^j D_jX^i-(D_jX^j)^2+R^{D}_{ji}X^iX^j.
    \label{Bochner equ 1}
\end{align} 
\end{prop}
\begin{proof} Using tensor calculus, we have
 \begin{align} D_i(X^jD_jX^i)=&D_iX^j D_jX^i+X^jD_iD_jX^i. \label{tensor calc}\end{align} 
 Combining \eqref{tensor calc} with the Ricci identity \eqref{Ricci identity} we get \eqref{Bochner1}.
 We also have  by using tensor calculus that 
 \begin{align}
 X^jD_jD_iX^i=& D_j(X^jD_iX^i)- (D_j X^j)(D_i X^i)
\nonumber\\=&D_i(X^iD_jX^j)- (D_i X^i)^2.
\label{tensor calc 2}\end{align}
Inserting \eqref{tensor calc 2} into \eqref{Bochner1}, we get \eqref{Bochner equ 1}. 
  \end{proof}

Our aim is to derive an integral formula  from the local Bochner formula \eqref{Bochner equ 1}. From now on,  let $(M^n, \bar g, V)$ be an $n$-dimensional Riemannian triple and $\a,\g\in \mathbb{R}$. Let   $D=D^{\a,\g}$ be the  affine connection defined as in (\ref{defi D}). %For simplicity, we use the abbreviation $D$ for $D^{\a,\g}$ when there is no confusion under the context.  
Note that $Ric^{D}$ is symmetric. In order to obtain a useful integral formula,  we need the following two important ingredients:
\begin{enumerate}
    \item [(i)] $DX$ is symmetric, i.e., 
 \begin{align}
     D_iX^j=D_jX^i;\end{align} 
    \item [(ii)] The left hand side of (\ref{Bochner equ 1}) needs to be  a ``divergent form" with respect to some choice of volume form.\end{enumerate}

In the following two lemmata, we will find an appropriate vector field $X$ and also a compatible volume form.

\begin{lemm}\label{sym vf}
Let $\phi$ be a smooth function on $M^n$. Let \begin{align} X=\bn^D\phi=V^{\g-\a}\nb \phi.\label{vf}
\end{align}
Then $$DX= V^{\g-\a} \left[\bn^2 \phi+\g du\otimes d\phi+\g d\phi \otimes du+ \a \bar g\left(\bn u, \bn \phi\right)\bar g\right]=\bn^{2,D}\phi$$ is symmetric. Moreover, $$\div^D X:=D_i X^i=V^{\g-\a} \left(\bar\De \phi+(2\g+n\a)\<\nb u, \nb \phi\>\right)=\bd^D \phi.$$
\end{lemm}

\begin{proof}
 Recall under local coordinates, 
\begin{align}
   % \nb_i\nb^j\phi=&\<\nb_i\nb \phi,\P_j\>\\
    D_iX^j=&\<D_i X,\bar g^{jk}\P_k\>.
    \label{hessian}
\end{align}
%We note under the same Riemannian metric the gradient vector fields $\nb \phi$ and $D \phi$ are the same.
Under normal coordinates, we have
\begin{align}
    D_iX_j=&\<D_i(V^{\g-\a}\bar\nabla\phi), \P_j\>\nonumber\\
    =&\nb_i(V^{\g-\a}\nb_j\phi)+\a V^{\g-\a}u_i\phi_j+\a V^{\g-\a}\nb u\nb \phi \delta_{ij}+\g V^{\g-\a}\phi_iu_j\nonumber\\
    =&V^{\g-\a}\left(\nb_i\nb_j\phi+\g u_i\phi_j+\g \phi_iu_j+\a\<\nb u, \nb \phi\> \delta_{ij}\right).
    \label{Bochner equ 2}
\end{align}
 Clearly,  $DX$ is symmetric.   
\end{proof}

\begin{lemm}\label{weight}
    Let $W$ be any smooth vector field on $M$. Then 
    \begin{align}
        V^\tau D_iW^i=\nb_i(V^\tau W^i)
        \label{div}
    \end{align}
where $\tau=(n+1)\a+\g$, is a divergent form with respect to the Riemannian volume form $d\O$. ($d\O$ denotes the  volume form induced by Riemannian metric $\bar g$ throughout this paper)
\end{lemm}
\begin{proof}
   By definition of $D$, we have
   \begin{align}
       D_iW^i=&\nb_iW^i+\a u_i W^i+\a u_k W^k \delta^i_i+\g W^i u_i\nonumber\\
       =&\nb_iW^i+[(n+1)\a+\g] u_i W^i.
       \label{lemm div equ 1}
   \end{align}
   Thus,
   \begin{align}
       V^\tau D_iW^i=& V^\tau\nb_iW^i+[(n+1)\a+\g] V^{\tau-1}V_i W^i\nonumber\\
= & V^\tau\nb_iW^i+\tau V^{\tau-1}V_i W^i\nonumber\\
= & \nb_i(V^\tau W^i).
       \label{lemm div equ 2}
   \end{align}
\end{proof}

As an immediate corollary, we can show that the volume form $V^\tau d\O$  is parallel under the new affine connection $D^{\a,\g}$. We thank the referee for pointing out this fact to us. The special case of $\a =\frac{1}{n-1}$ and $\g=0$ was proved in \cite{WY}.
\begin{coro}  We have $D_X (V^\tau d\O) =0$ for any smooth vector field $X$.
       
\end{coro}
\begin{proof} Given an arbitrary vector field $W$ with compact support in $\O$. From Lemma \ref{weight}, 
$$\int_\O W^i D_i (V^\tau d\O)= -\int_\O  V^\tau  D_iW^i d\O =-\int_\O \nb_i(V^\tau W^i) d\O=0.$$
Since $W$ is arbitrary, we conclude  that $D_i (V^\tau d\O)=0$ for any $i$.
   % The proof uses the fact that Riemannian volume form is parallel with respect to the Levi-Civita connection and Lemma \ref{weight}. We refer the direct computations to Proposition 3.2 in \cite{WY}.
\end{proof}

Choosing $X=\nb^D \phi$ in \eqref{Bochner equ 1}, we obtain
\begin{align}
   & D_i\Big((\nb^D \phi)^j D_j((\nb^D \phi)^i)-(\nb^D \phi)^i D_j((\nb^D \phi)^j)\Big)\nonumber \\=&|\bn^{2,D}\phi |_{\bar g}^2-|\bd^D \phi|^2+Ric^{D}(\bar\nabla^D\phi,\bar\nabla^D\phi).
    \label{Bochner equ 1'}
\end{align} 

Multiplying \eqref{Bochner equ 1'} with $V^\tau$ and integrating  over a bounded domain $\O\subset M$, we have% the following integral formula.
%\begin{prop}
%  Let $\phi$ be a smooth function on a bounded domain $\O\subset M$. We have
    \begin{align}
        &\di\int_{\O}V^\tau D_i\Big((\nb^D \phi)^j D_j((\nb^D \phi)^i)-(\nb^D \phi)^i D_j((\nb^D \phi)^j)\Big) d\O\nonumber\\
        =&\di\int_{\O}V^\tau \left[|\bn^{2,D}\phi |_{\bar g}^2-|\bd^D \phi|^2\right]+ V^\tau Ric^{D}(\bar\nabla^D\phi,\bar\nabla^D\phi) 
 d\O.
        \label{integral formula}
    \end{align}
%    where $m=\g -\a$ and $\tau=(n+1)\a+\g$.

%\end{prop}

%For a smooth hypersurface $\Sigma$ embedded in a Riemannian manifold $(M^n,g)$ with an affine connection $D$, we can define the affine mean curvature under the affine connection. Let $\nu$ be the unit outward normal vector filed of $\Sigma$.
%\begin{defi}
%  The affine second fundamental form $h^D$ is defined by $$h^D(X,Y)=\bar g(X,D_Y\nu)$$ for tangential vector fields $X$ and $Y$. The trace of the affine second fundamental form $h^D$ w.r.t. $\bar g$ is defined as the affine mean curvature and denoted as $H^D$. 
%\end{defi}

%\begin{lemm}
%    For a smooth hypersurface in a Riemannian manifold $M$, if $H$ and $H^D$ are the mean curvature under the Levi-Civita connection and the  affine connection $D^{\a,\g}$ respectively, then
 %   \begin{align}
 %       H^D=H+(n-1)\a\frac{V_\nu}{V}. 
 %       \label{mean curvature}
%    \end{align} 
%\end{lemm}

%\begin{proof}
%This follows by direct computation using the definition of $D^{\a,\g}$.
%\end{proof}

Applying (\ref{div}) and the Stokes' theorem on \eqref{integral formula}, we obtain the Reilly type integral formula. %We will use $\n$ and $\De$ to denote the gradient and Laplacian along the boundary $\S$.
\begin{theo}[Theorem \ref{theo RF0}] \label{theo RF}
Let $(M^n, \bar g, V=e^u)$ be an $n$-dimensional Riemannian triple and $\a,\g\in \mathbb{R}$. Let   $D=D^{\a,\g}$ be the  affine connection defined as in (\ref{defi D}) and $\tau=(n+1)\a+\g$. Let $\phi$ be a smooth function on a bounded domain $\O\subset M$ with smooth boundary $\S$. Then the following integral formula holds:
  %      \begin{align}
 %       &\di\int_{\O}V^\tau\Big(|\div^D (V^m \nb \phi)|^2-|D(V^m \nb \phi)|^2\Big) d\O\nonumber\\
 %       =&\di\int_{\S}V^{\tau+2m}\Big(H^D\phi_\nu^2+(h-\frac{V_\nu}{V}g)(\n\phi, \n \phi)-2V^{-\g}g(\n \phi, \n(V^\g \phi))\Big)dA\nonumber\\&+\int_{\O}V^{\tau}Ric^{D}(V^m \nb \phi,V^m \nb \phi)d\O.
 %       \label{formula RF 1}
 %   \end{align} 
         \begin{align*}       
        &\di\int_{\O}V^\tau \left[\left|\bd^D \phi\right|^2-\left|\bn^{2,D} \phi\right|_{\bar g}^2 \right] -V^\tau Ric^{D}( \nb^D \phi, \nb^D \phi) d\O \\ =&\di\int_{\S}V^{\tau}\left[H^D\<\bn^D \phi,\nu\>^2+\left(h-\g u_\nu g\right)(\n^D\phi, \n^D \phi)-2V^{-\g}\<\n^D \phi, \n^D(V^\g \phi_\nu)\>\right]dA.\nonumber
    \end{align*}

   %, $H^D$ is the affine mean curvature  defined in (\ref{mean curvature}).
\end{theo}

\begin{proof}
   % {\bf(Proof of Theorem \ref{theo RF}.)} 
   Using (\ref{div}), we have
    \begin{align*}       
        &\di\int_{\O}V^\tau D_i\Big((\nb^D \phi)^i D_j((\nb^D \phi)^j)-(\nb^D \phi)^j D_j((\nb^D \phi)^i)\Big) d\O\nonumber\\
        =&\di\int_{\O}\bar\nabla_i\Big[V^\tau\Big((\nb^D \phi)^i D_j((\nb^D \phi)^j)-(\nb^D \phi)^j D_j((\nb^D \phi)^i)\Big)\Big] d\O\nonumber\\
        =&\di\int_{\S}V^{\tau} \Big(\<\bn^D\phi,\nu\>D_j((\nb^D \phi)^j)-V^{\g-\alpha}\phi^j\<D_j(\nb^D \phi),\nu\>\Big) dA.
    \end{align*}
    Thus we only need to simplify the boundary term. 
    %where $\nu$ is the outward normal of the boundary hypersurface. 
    At any fixed point $P\in \S$, we choose normal coordinates with respect to $\bar g$ such that the indices $a=1,\cdots, n-1$ represents coordinates on $\S$ and $j=1,\cdots,n-1, \nu$ for coordinates on $\O$.
%$$\<\bn^D\phi,\nu\>D_j((\nb^D \phi)^j)-\phi^j\<D_j((\nb^D \phi)),\nu\>=V^m[ \phi_{\nu}D_j(V^m\phi_j)-\phi_{j}D_j(V^m\phi_\nu)].$$
    For simplicity, we will not distinguish upper and lower indexes and we denote $\phi_{ij}$ as the Hessian $\phi$ with respect to the Levi-Civita connection. 
    
    Using (\ref{Bochner equ 2}), we obtain
    \begin{align} 
       & \<\bn^D\phi,\nu\>D_j((\nb^D \phi)^j)
        =V^{2(\g-\a)}\big(\bar\De\phi\phi_\nu+(2\g+n\a)\<\bn u, \bn \phi\>\phi_\nu\big),\label{theo RF equ 2}
    \end{align}
    \begin{align}
        &V^{\g-\a}\phi^j\<D_j((\nb^D \phi)),\nu\>=V^{2(\g-\a)}\big(\phi_{j\nu}\phi_j+\g u_\nu\phi_j^2+(\a+\g) \<\bn u, \bn \phi\> \phi_\nu\big).
        \label{theo RF equ 3}
    \end{align}
    Combining (\ref{theo RF equ 2}) and (\ref{theo RF equ 3}), we have
    \begin{align} \label{theo RF equ 5}
    &\<\bn^D\phi,\nu\>D_j((\nb^D \phi)^j)-V^{\g-\a}\phi^j\<D_j(\nb^D \phi),\nu\>\\
       =&V^{2(\g-\a)}\Big(\phi_\nu\phi_{aa}-\phi_a\phi_{a\nu}+(n-1)\a u_\nu\phi_\nu^2+(\g+(n-1)\a) u_a\phi_a\phi_\nu-\g\phi_a^2u_\nu\Big)\nonumber\\
       =&V^{2(\g-\a)}\Big(H\phi_\nu^2+\phi_ah_{ab}\phi_b+\phi_\nu\De\phi-\phi_a\n_a\phi_\nu\nonumber\\
       &+(n-1)\a u_\nu\phi_\nu^2+(\g+(n-1)\a) u_a\phi_a\phi_\nu-\g\phi_a^2u_\nu\Big),\nonumber  
    \end{align}
where we used in the last equality the Gauss-Weigarten formula
\begin{align*}
    \phi_{aa}=\De\phi+H\phi_\nu,\qquad \phi_{a\nu}=\n_a\phi_\nu-h_{ab}\phi_b.
\end{align*}
Using $H^D=H+(n-1)\a u_\nu$ in \eqref{theo RF equ 5}, we get
    \begin{align} \label{theo RF equ 6}
        &\di\int_{\S}V^{\tau} \Big(\<\bn^D\phi,\nu\>D_j((\nb^D \phi)^j)-\phi^j\<D_j(\nb^D \phi),\nu\>\Big) dA\\
        =&\di\int_{\S}V^{\tau+2(\g-\a)}\Big(H^D\phi_\nu^2+(h-\g u_\nu g)(\n\phi, \n \phi)\Big)\nonumber\\&+\int_\S V^{\tau+2(\g-\a)}\Big(\phi_\nu\De\phi -\<\n\phi,\n\phi_\nu\>+ (\g+(n-1)\a)\< \n u,\n \phi\>\phi_\nu\Big)dA. \nonumber      
    \end{align}
 Integrating by parts for the last line of \eqref{theo RF equ 6} and noting $-[\tau+2(\g-\a)]+(\g+(n-1)\a)=-2\g$, we get the assertion.
\end{proof}

\

\section{Heintz-Karcher type and Minkowski type inequalities}\label{sec inequalities}

In this section, we will give proofs to the Heintz-Karcher type and Minkowski type inequalities stated in Theorem \ref{thm ineq}. \\

\noindent{\it Proof of Theorem \ref{thm ineq}  (i).}\
Recall that $\bd^D \phi=V^{\g-\a} \left[\bd \phi+(2\g+ n\a) \bar g\left(\bn u, \bn \phi\right)\right].$
We know from the standard elliptic PDE theory that the following Dirichlet boundary value problem \begin{equation}\label{Dirichlet3}
\left\{
    \begin{array}[]{rlll}
        \di \bd^D \phi&=&1&\mathrm{ in }\ \O,\\
        \phi&=& 0 &\mathrm{ on }\ \S,\\
    \end{array}
    \right.
\end{equation}
 admits a unique smooth solution $\phi\in C^{\infty}(\overline{\O})$. 
We will use the solution $\phi$ of the Dirichlet problem (\ref{Dirichlet3}) in \eqref{formula RF 1}. For $\tau=(n+1)\a+\g$, we have 
 \begin{align}\label{case2-1}
    \frac{n-1}{n}\int_\O V^\tau=&\frac{n-1}{n}\int_\O V^{\tau}(\bd^D \phi)^2\nonumber\\ 
    \geq&  \int_\O V^{\tau}\Big[|\bd^D \phi|^2-|\bn^{2,D}\phi|^2\Big]\nonumber\\
   \ge &\int_{\S} V^{\tau}H^D \<\bn^D \phi,\nu\>^2,
 \end{align}
 where in the first inequality, we have used the Cauchy-Schwarcz inequality, in the second inequality we have used integral formula (\ref{formula RF 1}), the nonnegativity of $Ric^{D}$, and  the Dirichlet boundary condition. 

 On the other hand,  using equation \eqref{Dirichlet3}, divergent structure \eqref{div}, and integration by parts,we have
\begin{eqnarray}\label{case2-2}
    \int_\O V^\tau &=&\int_{\O}V^{\tau} \bd^D \phi\\
    &=&\int_{\S} V^{\tau}\<\bn^D \phi,\nu\>. \nonumber
\end{eqnarray}
Combining \eqref{case2-1}, \eqref{case2-2} and using H\"older's inequality, we obtain
\begin{eqnarray}\label{equa1}
    \left(\int_\O V^\tau d\O\right)^2&=& \left(\int_\S V^{\tau}\<\bn^D \phi,\nu\> dA\right)^2\\
    &\leq& \int_\S V^{\tau}H^D\<\bn^D \phi,\nu\>^2 dA\int_\S \frac{V^\tau}{H^D}dA\nonumber
\\&\leq &\frac{n-1}{n}\int_\O V^\tau d\O\int_\S \frac{V^\tau}{H^D}dA.\nonumber
\end{eqnarray}
The assertion for the inequality follows. If the equality holds, we have 
\begin{align} \label{equa2}\bn^{2,D}\phi=\frac1n \bar g.\end{align}
Restricting \eqref{equa2} on $\S$, using $\phi=0$ on $\S$ and Gauss formula, we conclude that $h_{\a\b}=\l g_{\a\b}$ for some smooth function $\l$, i.e., $\S$ is umbilic.
\qed

\

\noindent{\it Proof of Theorem \ref{thm ineq} (ii).}
Consider the Neumann boundary value problem
   \begin{equation}
       \left\{
    \begin{array}[]{rlll}
        \di  \bd^D \phi&=&1&\mathrm{ in }\ \O,\\
       V^\g \phi_\nu&=& c &\mathrm{ on }\ \S,\\
    \end{array}
    \right.  
       \label{Neumann}
   \end{equation}
   where $c=\frac{\int_{\O}V^\tau}{\int_{\S}V^{\tau-\a}}$. The existence and uniqueness follows from the Fredholm alternative as in standard elliptic PDE theory. We will apply the solution $\phi$ of (\ref{Neumann}) to the integral formula (\ref{formula RF 1}). By using the Cauchy-Schwarz inequality, the equation and boundary condition in (\ref{Neumann}) and the curvature assumptions, we get
\begin{eqnarray*}
\frac{n-1}{n}\int_\O V^\tau d\O&\geq &\int_\S V^\tau H^D \<\bn^D\phi,\nu\>^2dA
\\&=& \int_\S V^\tau H^D V^{2\g-2\a}\phi_\nu^2dA
\\&=& c^2\int_\S V^{\tau-2\a} H^D dA.
\end{eqnarray*}
Inserting the value of $c$ we get the assertion.

 If the equality holds, we have 
\begin{align} \label{equa3}\bn^{2,D}\phi=\frac1n \bar g,\end{align}
and \begin{align} \label{equa4}(h-\g u_\nu g)(\nabla \phi,\nabla \phi)=0.\end{align}
Since by assumption $(h_{\a\b}-\g\frac{V_{\nu}}{V}g_{\a\b})>0$ on $\S$, it follows from \eqref{equa4} that $\phi=0$  on $\S$. 
Restricting \eqref{equa3} on $\S$, we see that $\S$ is umbilical.

\

\section{Poincar\'e type and Lichnerowicz type inequalities}\label{sec Lich}

Along the same line of the above results, we now prove a Poincare type inequality.  %{\color{blue}{In particular, even in the special case of Bakry-\'Emery Ricci curvature, our assumptions are weaker and results are strictly stronger than the known results. }}

%    \begin{theo}
 %        (Poincar\'e type inequalities) 
 %        Let $(M^n, \bar g, V)$ be an $n$-dimensional Riemannian triple. Let $\O$ be a bounded domain in $M$ with smooth boundary $\S$. Assume  the Ricci curvature with torsion  $Ric^{D,T}>0$. For any $f\in C^\infty(M)$ and  $\tau=\a+n\b+\g$, if one of the following alternatives holds,
%         \begin{itemize}
%         \item[(i)] $\S=\emptyset$ and $\di\int_{\O}fV^\tau d\O=0$;
%             \item [(ii)] $\S\neq\emptyset$, $f\equiv 0$ on $\S$ and $H^D\ge0$; 
%    \item [(iii)] $\S\neq \emptyset$, $\di\int_{\O}fV^\tau d\O=0$ and the second fundamental form of $\S$ satisfies   \begin{eqnarray}
%h_{ab}- \g\frac{V_{\nu}}{V} g_{ab}\geq 0.
%    \end{eqnarray}
%\end{itemize}
%        Then we have 
%                \begin{align}
%\frac{n}{n-1}\di\int_{\O}f^2V^\tau d\O\le \di\int_{\O}\left<\mathrm{Ric}^{-1}_{\mu,N}\n f,\n f\right>V^\tau d\O.
%            \label{Poincare case 1.0}
%        \end{align} 
%    \end{theo}

\begin{theo}\label{Poincare}
 Let $(M^n, \bar g, V=e^u)$ be an $n$-dimensional Riemannian triple and $\a,\g\in \mathbb{R}$. Let $\O$ be a bounded domain in $M$ with smooth boundary $\S$. Assume  $Ric^{D}$ of $D=D^{\a, \g}$ is positive definite. For any $f\in C^\infty(M)$ and  $\tau=(n+1)\a+\g$, if one of the following alternatives holds,
         \begin{itemize}
         \item[(i)] $\S=\emptyset$ and $\di\int_{\O}f V^\tau d\O=0$;
             \item [(ii)] $\S\neq\emptyset$, $f\equiv 0$ on $\S$ and $H^D\ge0$; 
    \item [(iii)] $\S\neq \emptyset$, $\di\int_{\O}f V^\tau d\O=0$ and  $\S$ satisfies $h-\g u_\nu g\ge 0$.   
         \end{itemize}
        Then we have 
                \begin{align}
\frac{n}{n-1}\di\int_{\O}f^2V^\tau d\O\le \di\int_{\O}\left<{(\mathrm{Ric}^{D})}^{-1}\nb f,\nb f\right>V^\tau d\O.
            \label{Poincare case 1}
        \end{align} 
\end{theo}

\begin{proof}
The proof is similar as in \cite{Mi} while we use $Ric^D$ in this paper.  In case (i), we solve PDE
            \begin{eqnarray}\label{closed Poincare}
\bd^D \phi=f \hbox{ in }\O.
\end{eqnarray}
            In case (ii), we solve the Dirichlet boundary value problem below,
\begin{equation}\label{Dirichlet Poincare}
\left\{
    \begin{array}[]{rlll}
        \di \bd^D \phi&=&f&\mathrm{ in }\ \O,\\
        \phi&=& 0 &\mathrm{ on }\ \S.\\
    \end{array}
    \right.
\end{equation}
In case (iii), we solve the Neumann boundary value problem 
\begin{equation}\label{Neumann Poincare}
\left\{
    \begin{array}[]{rlll}
        \di \bd^D \phi&=&f&\mathrm{ in }\ \O,\\
        \phi_{\nu}&=& 0 &\mathrm{ on }\ \S.\\
    \end{array}
    \right.
\end{equation}
Problems \eqref{closed Poincare} and \eqref{Neumann Poincare} are solvable since $\di\int_{\O}fV^\tau d\O=0$.

%Note that by (\ref{div}), $\di\int_{\O}V^{\tau+m} \phi_{\nu} d\mu=\di\int_{\O}fV^{\tau}d\mu=0$, so a solution exists. 

In all these three cases, we apply the integral formula (\ref{formula RF 1}) to the solutions of the PDEs, i.e. $\phi$ satisfying $\bd ^D \phi =f$. By Cauchy-Schwarz inequality, we have
\begin{equation} 
    \begin{array}[]{rll}
&    \di\frac{n-1}{n}\int_{\O}f^2V^\tau - \int_{\O}V^\tau Ric^{D}(\bn^D \phi, \bn^D \phi),\\ 
\ge&\di\int_{\S}V^{\tau}\left[H^D\<\bn^D \phi,\nu\>^2+\left(h-\g u_\nu g\right)(\n^D\phi, \n^D \phi)-2V^{-\g}\<\n^D \phi, \n^D(V^\g \phi_\nu)\>\right]dA.\label{equ extra 1} 
    \end{array} 
\end{equation} 
where the right hand side contains only boundary integrations. Next, we will show that in all these cases, the right hand side boundary integrations are nonnegative. 

In case (i), the boundary $\S=\emptyset$ and the result follows immediately. 

In case (ii), the first boundary integral in (\ref{equ extra 1}) is nonnegative since $H^D\ge0$. Recall from Definition \ref{notation} (i), $\n^D \phi=V^{\gamma -\alpha}\n \phi$ where $\n $ is the covariant derivative with respect to the induced metric of the boundary $\S$. Thus we have $\n^D\phi \equiv 0$, since $\phi\equiv 0$ on $\S$. 

In case (iii), $\phi_\nu\equiv 0$ on $\S$. We observe that $\left<\bar\n^D\phi, \nu\right>=V^{\g-\alpha}\phi_\nu\equiv 0$ and $\n^D(V^\g\phi_\nu)=V^{\g-\alpha}\n(V^\g\phi_{\nu})\equiv0$. Thus the first and last boundary terms in (\ref{equ extra 1}) are all zero. Under the condition that $h-\g u_\nu g\ge0$, we conclude that the right hand side boundary terms are all nonnegative. 

Equivalently, we have shown, in all the three cases, the following inequality holds, 
\begin{align}
    \di\frac{n-1}{n}\int_{\O}f^2V^\tau  \ge \int_{\O}V^\tau Ric^{D}(\bn^D \phi, \bn^D \phi).
    \label{Neumann Poincare equ 1}
\end{align}

%where we have also used the Cauchy-Schwarz inequality.
On the other hand, by integration by parts and recall Definition \ref{notation}, also noting the divergence property \eqref{div}, we have
\begin{align}
    \begin{array}[]{rll}
    \di\int_{\O}f^2V^\tau  =&\di\int_{\O}fV^\tau \bd^D \phi=\int_{\O} f\bar\n(V^{\tau}\bar\n^D\phi )\\
\\
    =&\displaystyle\int_{\S} fV^{\tau}\bar\n^D_\nu\phi   -\int_{\O}V^{\tau}\<\bn^D \phi, \bar \n  f\>\\
    \\
    =&\displaystyle-\int_{\O}V^{\tau}\<\bn^D \phi, \bar \n  f\>,\\
    \end{array}
    \label{Neumann Poincare equ 1.2}
\end{align}
where the last identity holds due to the following observations: case(i), $\S=\emptyset$; case (ii), $f\equiv 0$ on $\S$; case (iii), $\phi_\nu\equiv 0$ on $\S$.

Applying H\"older's inequality to (\ref{Neumann Poincare equ 1.2}), we obtain
\begin{eqnarray}    \label{Neumann Poincare equ 2}
    \di&&\left(\int_{\O}f^2V^\tau \right)^2 \le \left(\di\int_{\O}V^\tau Ric^{D}(\bn^D \phi,\bn^D\phi)\right)\left(\int_{\O}V^\tau \left<{(Ric^{D})}^{-1}\nb f,\nb f\right>\right).
\end{eqnarray}
Combining (\ref{Neumann Poincare equ 1}) and (\ref{Neumann Poincare equ 2}), we proved \eqref{Poincare case 1}.
\end{proof}

\

As a consequence, we get the Lichnerowicz type inequality for the first eigenvalue of $D$-Laplacian $\bd^D$. %It is direct to compute that
%$$\mathrm{div}^D(V^{m}\bar\nabla\phi)=V^{m}\bar\De \phi+(2\g+n\a) V^{m-1}\nb V\nb \phi.$$

\noindent{\it Proof of Theorem \ref{thm ineq} (iii).} 
By using the divergence property \eqref{div}, One sees directly that the first eigenvalues have the following variational representation
  \begin{align}
             \lambda_1^{Dir} =\inf_{\{f\in C^1(\overline \O); f\not\equiv 0, f|_{\S}=0\}}\di\frac{\int_\O|\bar \n f|^2V^{\tau+\g-\a} d\O}{\int_M f^2V^\tau d\O},
             \label{eigenvalue}
         \end{align}
  \begin{align}
             \lambda_1^{Neu} =\inf_{\{f\in C^1(\overline \O); f\not\equiv 0, \int_\O f V^\tau=0\}}\di\frac{\int_M|\bar \n f|^2V^{\tau+\g-\a} d\O}{\int_M f^2V^\tau d\O}.
             \label{eigenvalue Neumann}
         \end{align}
         The assertion follows immediately from Theorem \ref{Poincare} above by using the curvature condition $Ric^D\ge (n-1)V^{\a-\g}\bar g$.
\qed

\

\noindent{\it Proof of Corollary \ref{coro}.}  Let $\a=0, \g=1$. It is sufficient to observe the following fact: the Dirichlet boundary problem for PDEs 
\begin{equation}
\left\{
    \begin{array}[]{rlll}
        \di \bd^D\phi &=&-\lambda_1^{Dir}
        \phi&\mathrm{ in }\ \O,\\
        \phi&=& 0 &\mathrm{ on }\ \S,\\
    \end{array}
    \right.
\end{equation}
and
\begin{equation}
\left\{
    \begin{array}[]{rlll}
        \di V\bar \De f-\bar \De Vf&=&-\lambda_1^{Dir}
       f&\mathrm{ in }\ \O,\\
        f&=& 0 &\mathrm{ on }\ \S,\\
    \end{array}
    \right.
\end{equation} are equivalent under the correspondence $\phi=\frac{f}{V}$.
The same phenomenon holds for Neumann boundary problems. Then Corollary \ref{coro} follows from Theorem \ref{thm ineq} (iii).
\qed

\

\section{Appendix: Proof of Proposition \ref{thm Ricci}}\label{appenA}

In this appendix we prove Proposition \ref{thm Ricci}. We only need to prove it at any fixed point $P$ under  normal coordinates system $\{\P_i\}_{i=1}^n$ such that $\bar g_{ij}(P)=\delta_{ij}$, $\P_k\bar g_{ij}(P)=0$, and $\nb_{\P_i}\P_j(P)=0$. Throughout the proof we will use these properties implicitly. We use $u_i:=du(\P_i)$ and let $u_{ij}$ be the Hessian of $u$ with respect to $\nb$. 

    By definition, we first observe that
    \begin{align}
        D_i\P_j=\nb_i\P_j+\a u_i\P_j+\a u_j\P_i+\g \d_{ij}\nb u,
        \label{thm Ricci equ1}
    \end{align}
    and 
    \begin{align}
        D_k(D_i\P_j)=&\nb_k(D_i\P_j)+\a u_k(D_i\P_j)+\a du(D_i\P_j)\P_k+\g \bar g(\P_k,D_i\P_j)\nb u.
        \label{thm Ricci equ2}
    \end{align}
    Next, we compute each term on the right of (\ref{thm Ricci equ2}). It turns out
    \begin{align}
        \nb_k(D_i\P_j)=&\nb_k(\nb_i\P_j)+\a u_{ik}\P_j+\a u_{jk}\P_i+\g \d_{ij}\nb_k\nb u,\nonumber\\
        \a u_k (D_i\P_j)=&\a^2u_ku_i\P_j+\a^2 u_ku_j\P_i+\a\g \d_{ij}u_k\nb u,\nonumber\\
        \a du(D_i\P_j)\P_k=& 2\a^2u_iu_j\P_k+\a\g \d_{ij}|\nb u|^2\P_k,\nonumber\\
        \g \bar g(\P_k,D_i\P_j)\nb u=&(\a\g u_i \d_{jk}+\a\g u_j\d_{ik})\nb u +\g^2\d_{ij}u_k\nb u.
        \label{thm Ricci equ3}
    \end{align}
Adding up these four terms into \eqref{thm Ricci equ2}, we have
\begin{align}
      D_k(D_i\P_j)=&\nb_k(\nb_i\P_j)+\a u_{ik}\P_j+\a u_{jk}\P_i+\g \d_{ij}\nb_k\nb u\nonumber\\
&+\a^2u_ku_i\P_j+\a^2 u_ku_j\P_i+\a\g \d_{ij}u_k\nb u\nonumber\\
 &+ 2\a^2u_iu_j\P_k+\a\g \d_{ij}|\nb u|^2\P_k\nonumber\\
  &+(\a\g u_i \d_{jk}+\a\g u_j\d_{ik})\nb u +\g^2\d_{ij}u_k\nb u.
\label{thm Ricci equ4}
\end{align}
Using the metric tensor $\bar g$, we have
\begin{align}
    \left< D_k(D_i\P_j),\P_k\right>_{\bar g}=&\left< \nb_k(\nb_i\P_j),\P_k\right>_{\bar g}+2\a u_{ij}+\g \bar\De u \d_{ij}\nonumber\\
    &+\big[\a^2+(2n+1)\a^2+2\a\g\big]u_iu_j\nonumber\\
    &+\big[(n+1)\a\g+\g^2\big]|\nb u|^2 \d_{ij}.
    \label{thm Ricci equ5}
\end{align} 
By swapping $k$ and $i$ in (\ref{thm Ricci equ4}), we obtain
\begin{align}
D_i(D_k\P_j)=&\nb_i(\nb_k\P_j)+\a u_{ik}\P_j+\a u_{ji}\P_k+\g \d_{kj}\nb_i\nb u\nonumber\\
&+\a^2u_ku_i\P_j+\a^2 u_iu_j\P_k+\a\g \d_{kj}u_i\nb u\nonumber\\
 &+ 2\a^2u_ku_j\P_i+\a\g \d_{kj}|\nb u|^2\P_i\nonumber\\
  &+(\a\g u_k \d_{ij}+\a\g u_j\d_{ik})\nb u +\g^2\d_{kj}u_i\nb u.
  \label{thm Ricci equ4.1}
\end{align}
Similarly, we have
\begin{align}
    \left< D_i(D_k\P_j),\P_k\right>_{\bar g} =&     \left< \nb_i(\nb_k\P_j),\P_k\right>_{\bar g}+\big[(n+1)\a+\g\big]u_{ij}+2\a\g|\nb u|^2\d_{ij}\nonumber\\ 
   &+\big[(n+3)\a^2+2\a\g+\g^2\big]u_iu_j.
  \label{thm Ricci equ5.1}
\end{align}
Combining (\ref{thm Ricci equ5}) and (\ref{thm Ricci equ5.1}), using Definition \ref{defi Ricci} we have
\begin{align}
  R^D_{ij}= &\left< R^D(\P_k,\P_i)\P_j,\P_k\right>_{\bar g} \nonumber\\ 
    =&    \left< D_k(D_i\P_j),\P_k\right>_{\bar g}   - \left< D_i(D_k\P_j),\P_k\right>_{\bar g} \nonumber\\
    =&R_{ij}-\big[\a(n-1)+\g\big]u_{ij}+\big[(n-1)\a^2-\g^2\big]u_iu_j\nonumber\\
&+\g \bar\De u\d_{ij}+\big[(n-1)\a\g+\g^2\big]|\nb u|^2\d_{ij}.
    \label{thm Ricci equ 6}
\end{align}
   This finishes the proof for $Ric^D$. \\ 
\qed
\

\bigskip
{\bf Acknowledgement:} Part of the work was done during both authors'€˜ visit at McGill University in Fall, 2015. We would like to thank the department for the hospitality and professor Pengfei Guan for constant support. The authors also would like to thank the referee for his critical reading and  valuable suggestions.

\

\bibliographystyle{amsplain}

\end{document}